\documentclass[12pt]{amsart}
\usepackage{amsmath,amssymb,amsfonts,amsthm,amstext}
\usepackage{url,xspace}
\usepackage{color,graphicx}
\usepackage{enumerate}

\usepackage[margin=3cm]{geometry}

\usepackage{floatrow}

\newcommand{\vol}{\operatorname{vol}}

\newcommand{\old}[1]{}

\newcommand{\Z}{{\mathbb Z}}
\newcommand{\R}{{\mathbb R}}
\newcommand{\N}{{\mathbb N}}

\renewcommand{\P}{{\mathbb P}}
\newcommand{\E}{{\mathbb E}}

\newcommand{\cE}{\mathcal{E}}

\newcommand{\trho}{\widetilde{\rho}}
\newcommand{\tell}{\widetilde{\ell}}

\newtheorem{thm}{Theorem}
\newtheorem{lemma}{Lemma}
\newtheorem{prop}[lemma]{Proposition}

\theoremstyle{definition}

\newtheorem{remark}[lemma]{Remark}

\numberwithin{lemma}{section}
\numberwithin{equation}{section}

\newcounter{mycount}

\title[Stable matchings via the PWIT]{Stable matchings in high dimensions via the Poisson-weighted infinite tree}
\author{Alexander E.\ Holroyd, James B.\ Martin and Yuval Peres}



\date{27 December 2018}

\begin{document}
\maketitle

\begin{abstract}
We consider the stable matching of two independent Poisson processes in $\R^d$ under an asymmetric color restriction.  Blue points can only match to red points, while red points can match to points of either color.   It is unknown whether there exists a choice of intensities of the red and blue processes under which all points are matched.  We prove that for any fixed intensities, there are unmatched blue points in sufficiently high dimension.  Indeed, if the ratio of red to blue intensities is $\rho$ then the intensity of unmatched blue points converges to $e^{-\rho}/(1+\rho)$ as $d\to\infty$.  We also establish analogous results for certain multi-color variants.  Our proof uses stable matching on the Poisson-weighted infinite tree (PWIT), which can be analyzed via  differential equations.  The PWIT
has been used in many settings as a scaling limit
for models involving complete graphs with independent
edge weights,
but as far as we are aware, this is the first
presentation of a rigorous application to high-dimensional Euclidean space.  Finally, we analyze the asymmetric matching problem under a hierarchical metric, and show that there are unmatched points for all intensities.
\end{abstract}


\section{Introduction}
\subsection{Stable matching in $\R^d$}
Let $S$ be a set of points in $\R^d$,
and let $M$ be a matching of $S$.
For each $x\in S$, write
$d_M(x)$ for the distance of $x$ from its partner
in the matching $M$, with $d_M(x)=\infty$ if $x$ is
unmatched by $M$.

A pair of distinct points $x,y\in S$
is an \textbf{unstable pair} for the matching $M$
if $|x-y|<\min(d_M(x), d_M(y))$.
If $M$ has no unstable pairs,
then it is said to be
a \textbf{stable matching}
(as introduced by Gale and Shapley \cite{GaleShapley}).
We can interpret this definition by
saying that each point would like to find a partner
as near as possible to itself (and would prefer any partner
rather than remaining unmatched), and that a matching $M$ is stable
if there is no pair of points that would both prefer to be matched
to each other over their situation in $M$.

Holroyd, Pemantle, Peres and Schramm \cite{HPPS}
studied stable matching for the points of a
homogenous
Poisson process. They showed
that with probability 1, there exists
a unique stable matching, under which every point is matched.
Let the random variable $X$ represent the distance of
a typical point of the process to its partner in this stable matching. Then Theorem 5 of \cite{HPPS} says that $\E X^d=\infty$,
but $\P(X>r)\leq Cr^{-d}$ for some constant $C=C(d)<\infty$.

Now suppose that the points of the process are of two types;
each point is independently coloured blue with probability
$p \in(0,1)$ and red with probability $1-p$.
Restrict to matchings in which a red point and a blue
point may be matched, but two points of the same colour may
not be matched.
Correspondingly
the definition of unstable pair is restricted to
pairs consisting of one red point and one blue point; the definition of a
stable matching is otherwise unchanged.
Again, it is shown
in \cite{HPPS} that with probability 1 there exists a unique stable matching.
If $p=1/2$ (so that the model is symmetric between red and blue)
then with probability 1, every point is matched;
it is shown that the distribution of the distance
from a typical point to its partner has a polynomial tail
(although for $d\geq 2$ there is a gap between the upper and lower
bounds on the exponent). On the other hand, suppose $p<1/2$. Then
with probability 1, all blue points are matched, and a positive
density of red points remain unmatched.

%
%
\begin{figure}
\begin{center}
\includegraphics[width=0.68\textwidth]
{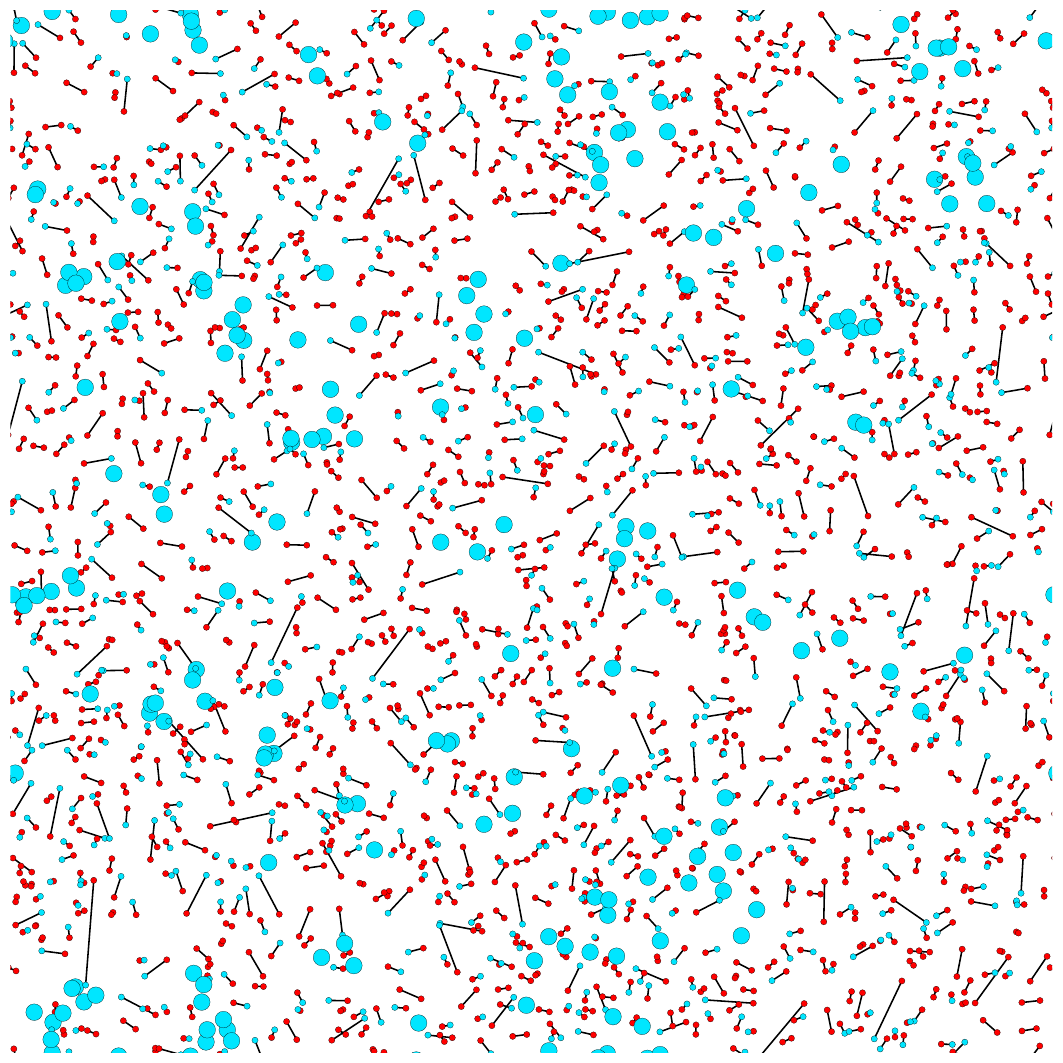}
\\[1ex]
\includegraphics[width=0.68\textwidth]
{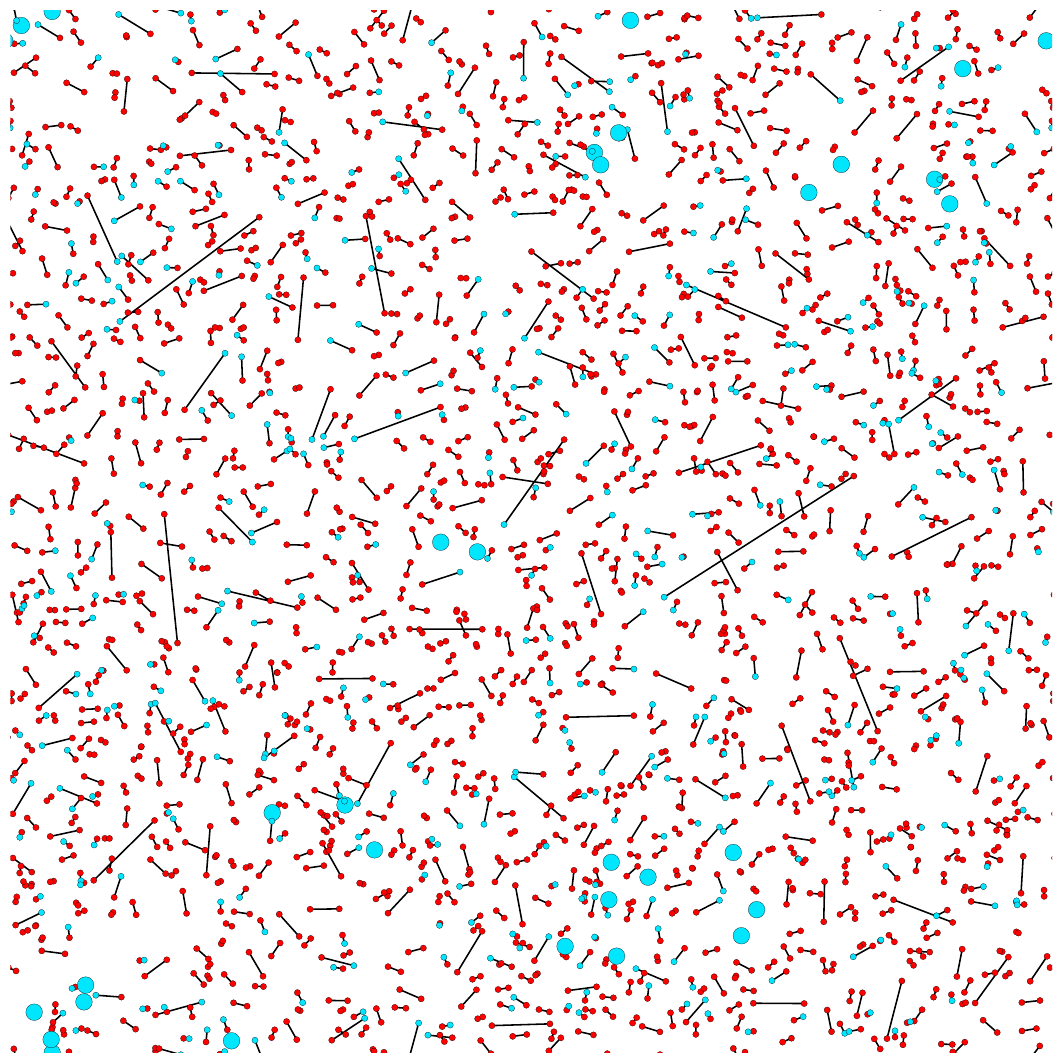}
\caption{
\label{fig:asymm}
The two-color asymmetric model
(with red-red and red-blue matches allowed)
for uniformly distributed points in a two-dimensional torus.
Unmatched points are shown larger.  \emph{Top:} 2000 red and 1000 blue points.  \emph{Bottom:} 2500 red and 500 blue points.
}
\end{center}
\end{figure}

For the models above, the question of whether
the stable matching is perfect (i.e.\ whether every point is matched)
is easy to answer using arguments involving translation invariance,
ergodicity and mass transport (although many interesting questions
remain about the nature of the matching).
In this paper we study natural variants where, in contrast, the question of
whether the stable matching is perfect already
presents a challenge.

Again suppose the points of a Poisson process in $\R^d$ are colored
independently blue (with probability $p$) or red (with probability $1-p$).
We now consider an asymmetric rule, under which red-red and red-blue
matches are allowed, while blue-blue matches are forbidden.
Subject to this restriction, each point prefers to be matched
at as short a distance as possible.
From Proposition \ref{prop:uniquestable} below,
we will be able to obtain that with probability 1 a stable matching
exists and is unique. Let $M$ be this stable matching. From the ergodicity of the Poisson process,
the intensity of
the set of red points matched by $M$ to blue points
is an almost sure constant, and the same is true for the
set of blue points matched by $M$ to red points. By a mass transport argument,
these intensities are equal.

If $p>1/2$, then there must be some blue points left unmatched.
In fact, it is easy to see that this is still the case for some $p<1/2$,
since some pairs of red points will be matched to each other
(for example, any pair which are each other's nearest neighbours).
We conjecture that this remains true for all $p>0$. As far as we are aware,
this is not known for any $d$. Here we make the following progress towards the conjecture:
for any fixed $p>0$, if $d$ is sufficiently large, then there are unmatched blue points.
%
%

\begin{thm}
\label{thm:asymmetricRd}
For a Poisson process of intensity $1$ in $\R^d$ in which each point independently is blue with probability $p\in(0,1)$ and red with probability
$1-p$, consider the asymmetric two-type stable matching, under which only
red-red and red-blue matches are allowed.  For fixed $p$, the intensity of
unmatched blue points
converges to $p e^{-(1-p)/p}$ as $d\to\infty$.
For a non-zero density of unmatched blue points, it suffices to take $d>\frac{c}{p}e^{1/p}$,
where $c$ is some absolute constant.
\end{thm}

See Figure \ref{fig:asymm} for simulations of the asymmetric two-type model
in a two-dimensional torus.  The two simulations use the same random set of points, but some blue points in the upper picture become red in the lower picture.  Can the reader find a blue point that is matched in the upper picture but unmatched in the lower one?

We next consider a multi-type symmetric model.
Suppose that there are $k$ different colours,
and each point of the Poisson process independently
receives colour $i$ with probability $p_i$, where $(p_1, \dots, p_k)$ is a probability vector.

Any two points of different colours can be matched,
but two points of the same colour may not be matched.

Again, Proposition \ref{prop:uniquestable} will yield that
there exists a unique stable matching. If $p_1>1/2$,
then with probability 1, some points of colour 1 will remain
unmatched in this stable matching.
We conjecture that this remains true whenever
$p_1>\max_{2\leq i\leq k}p_i$.
We show that
for a given collection $p_1,\dots,p_k$, this
is true for sufficiently large $d$.

\begin{thm}
\label{thm:symmetricRd}
Fix a probability vector $(p_1,\dots, p_k)$.
Consider the multi-type symmetric stable matching
of a Poisson process of rate 1 in $\R^d$
with $k$ colours,
where $p_i$ gives the probability of colour $i$.
Suppose $p_1>\max_{2\leq i\leq k}p_i$.
Then there exists some strictly positive $\lambda=\lambda(p_1,\dots,p_k)$
such that the intensity of unmatched type-1 points
converges to $\lambda$ as $d\to\infty$.
In the case where $p_i$ are equal for all $i\geq 2$, we have
$\lambda=(p_1-p_2)^{k-1} p_1^{-(k-2)}$.
\end{thm}

Our method to prove Theorems \ref{thm:asymmetricRd} and \ref{thm:symmetricRd}
involves the analysis of stable matchings on the
\textit{Poisson-weighted infinite tree} (or \textit{PWIT}).
The PWIT was introduced by Aldous and Steele \cite{AldousSteele}
and has been used in many contexts to provide a scaling limit
of complete graphs with independent edge-weights,
for example in the setting of
minimal-weight spanning trees and invasion percolation
\cite{ABGriffKang,AldousSteele,SteeleMST},
random assignment problems \cite{AldousRA, SalezShah},
and random matrices \cite{BCC1, BCC2, BordenaveChafai}.
Here we show how it also arises naturally as a limit
of Poisson processes in high-dimensional Euclidean space,
under appropriate rescaling; as far as we know,
this is the first such application of the PWIT.

We make some brief observations illustrating some of the difficulty of obtaining results about the models considered in Theorems \ref{thm:asymmetricRd}
and \ref{thm:symmetricRd}.

In the symmetric model of Theorem \ref{thm:symmetricRd},
note that for any $i$, the probability that there
are unmatched points of type $i$ is 0 or 1, by ergodicity. If $p_i=p_j$
then by symmetry this probability is the same for $i$ and for $j$, but there can't be unmatched points of two different colours,
so the probability must be 0. Consider for example $k=3$ and $p_1<p_2=p_3$. By the above argument, with probability 1 there are no unmatched points of type $2$ or $3$. One would naturally conjecture that also no points of type $1$
(which has lower intensity) are unmatched -- however,
there is no obvious monotonicity for the model in Euclidean space, and this
conjecture does not seem easy to prove, even taking $d$ large. (Our comparison with the PWIT could be used
to show that the intensity of unmatched points of type 1 can be made
as small as desired by taking $d$ sufficiently large, but it is not clear that it will help in showing that the density is 0 for some $d$.)

Meanwhile for the asymmetric two-type model, Theorem \ref{thm:asymmetricRd}
implies that for given $p_0>0$,
for sufficiently large $d$ there exist
unmatched blue points with probability 1
for any probability $p>p_0$ of blue points.
One might naturally imagine that for any fixed $d$,
if there are unmatched blue points for some given $p_0$,
then the same is true for any $p>p_0$.
However, this monotonicity property also appears difficult to prove.
(Note that one can easily find finite
configurations of points such that removing a
blue point increases the number of unmatched blue points.)

Underlying both 
Theorem \ref{thm:asymmetricRd} and Theorem \ref{thm:symmetricRd}
is a more general property, namely
that for the model of a stable matching
of a rate-$1$ Poisson process in
$\R^d$ with a given 
distribution of colours and a given rule about which 
colour-pairings are allowed,
the intensity of unmatched points of a given colour
converges as $d\to\infty$;
its limit can be identified 
as the probability,
in a corresponding stable matching model 
on the PWIT,
that the root has the  
given colour and is unmatched. 
We state
this more general result after we have given
a formal definition of the PWIT and the stable matching model
on it; see Theorem \ref{thm:meta} at the end of 
Section \ref{sec:conclusionRd} for the formulation. 

\subsection{Asymmetric two-type matching on a
hierarchical graph}\label{subsec:hierarchical}

We also consider the asymmetric two-type model in a case where
the distance function is given by a hierarchical metric.
In this case we can indeed show that there are unmatched points
for every value of $p$ (as we conjectured above for the stable matching
with respect to Euclidean distance in $\R^d$).

Consider the distance on $\R_+$ defined by
\begin{equation}\label{rhodef}
\rho(x,y)=2^{-\sup\{k\in\Z:\lfloor2^{k}x\rfloor=\lfloor2^{k}y\rfloor\}}.
\end{equation}
This is an ultrametric (that is, a metric such that
$\rho(x,z)\leq\max\{\rho(x,y), \rho(y,z)\}$ for all $x,y,z$).

Let $p\in(0,1)$.
Consider a Poisson process of rate $\lambda>0$ on $\R_+$.
As above, let each point independently be coloured
blue with probability $p$ and red with probability $1-p$;
red-red and red-blue matches are allowed, but not blue-blue.

\begin{thm}
\label{thm:hierarchical}
Fix $\lambda>0$ and $p\in(0,1)$.
Consider the two-type asymmetric stable matching model for
a Poisson process on $\R_+$ with rate $\lambda$,
in which each point is independently blue with
probability $p$ and
red with probability $1-p$.
With probability 1, there
are infinitely many stable matchings
with respect to the hierarchical metric $\rho$,
and all these matchings have infinitely many
unmatched blue points.
In fact, there exists $c=c(\lambda, p)>0$
such that, with probability 1, for all large enough $R$
the number of unmatched blue points
in the interval $[0,R]$ is at least $cR$ for
all stable matchings.
\end{thm}

There exist multiple stable matchings with respect to $\rho$
since a point can be equidistant from several others.
However, these stable matchings are all closely related
to each
other in the following way.
Take a dyadic interval $[2^km,2^k(m+1))$ for some integers $m\geq0$ and $k$,
and, for a given stable matching,
consider the set of points in the interval that are
not matched to another point in the interval
(equivalently, that are not matched at distance $2^k$ or less).
This set cannot include both a red point and a blue point,
and also cannot include two red points.
We write $N_k(m)$ for the number of blue points minus the number
of red points in the set, which is in $\{-1,0,1,2,\dots\}$.

Then we have
\begin{equation}
\label{hierarchicalrecursion}
N_k(m)=g\left(N_{k-1}(2m), N_{k-1}(2m+1)\right)
\end{equation}
where $g:\{-1,0,1,2,\dots\}^2 \mapsto \{-1,0,1,2,\dots\}$
is defined by
\[
g(a,b)=
\begin{cases}0&\text{if }a=b=-1\\a+b&\text{otherwise}\end{cases}.
\]
Also $N_k(m)$ is uniquely determined whenever the
interval has at most one point. By starting from a partition into dyadic intervals each of which contain at most one point
and applying (\ref{hierarchicalrecursion}) recursively, we obtain that the values $N_k(m)$ are in fact the same for any stable matching.

By translation invariance, $N_k(m), m\in\Z$ are i.i.d.\
for every $k$, and (\ref{hierarchicalrecursion}) yields a recursion
in $k$ for the distribution of $N_k(m)$, which we analyse
in order to prove Theorem \ref{thm:hierarchical}.

One could also consider a version in $\R^d$ for $d\geq 2$
based on dyadic $d$-cubes rather than dyadic intervals,
or indeed a more general model $p$-adic model for all $p\geq 2$.
The result and the analysis would be very similar.

\subsection{Related work}

The recent article \cite{amir-angel-holroyd} treats general (not necessarily stable)
\linebreak[4]
translation-invariant matchings of Poisson processes of multiple colors under arbitrary color-matching rules (and indeed generalized matchings in which three or points may be matched to each other).  The optimal tail behavior of such matchings in $\R^d$ is analyzed in terms of the dimension $d$, the matching rule, and the color intensities.  It turns out that convex geometry in the space of intensity vectors plays a key role.

In \cite{frogs}, stable matchings of various kinds are shown to be intimately tied to two-player games.  In particular, the asymmetric matching rule of Theorem~\ref{thm:asymmetricRd} is related to a game called ``fussy friendly frogs'': Alice places a frog on a point of the Poisson process, then Bob places another frog on a distinct point; subsequently the players take turns to move either frog in such a way that the two frogs get closer, but they are never allowed to be both on blue points or on the same point; a player with no legal move loses.  Theorem~\ref{thm:asymmetricRd} implies that this game is a first-player win in sufficiently high dimension.

\subsection{Plan of the paper}
In Section \ref{sec:uniqueness} we set up
a formal framework for stable matchings on weighted graphs,
and give a result which guarantees existence and
uniqueness of the stable matching for several of the
multi-type models that we consider in the paper. The
required conditions are that the weights involving
any given vertex are distinct and have no accumulation points, and that there are no
infinite \textit{descending paths}
(in the sense of \cite{DaleyLast}).

In Section \ref{sec:PWIT} we describe the PWIT
and explain how it arises as a scaling
limit of high-dimensional Poisson processes. We then
investigate various stable matching models on the PWIT
(where the recursive structure of the graph makes
certain exact computations possible).

The comparison between the stable matching models for the PWIT
and for the Poisson process in $\R^d$ is formally developed in
Section \ref{sec:coupling}.
In Sections \ref{sec:conclusionRd}, \ref{sec:conclusionhierarchical}
and \ref{sec:stableproof} we complete the proofs
of the main results.  We conclude the article with some open problems.

\section{Stable matching on general weighted graphs}
\label{sec:uniqueness}
In this section we give a result guaranteeing the
existence and uniqueness of a stable matching
for a general class of models based on a symmetric
distance function (which need not be a metric),
or, in other words, an edge-weighted graph.
In particular, the framework will
cover the various multi-type models considered above.
(A pair of points whose types are incompatible
will not be joined by an edge in the graph.)

Suppose we have a set $V$ and a symmetric function
$\ell:V\times V\to \R_+\cup\{\infty\}$
with $\ell(x,x)=\infty$ for all $x$.
We call $\ell(x,y)$ the \textbf{weight} associated
to the pair $\{x,y\}$;
we think of it as the weight of the edge $\{x,y\}$
in a weighted graph with vertex set $V$
and vertex set $E_\ell:=
\big\{\{x,y\}:\ell(x,y)<\infty\big\}$,
with the case $\ell(x,y)=\infty$
corresponding to the absence of an edge.

Let a \textbf{matching} of $(V,E_\ell)$
be a function $M:V\to V$ that is an involution,
i.e.\ $M(M(x))=x$ for all $x\in V$,
and such that $\ell(x,M(x))<\infty$ whenever
$M(x)\ne x$.

If $x\ne y$ and $M(x)=y$ (in which case also $M(y)=x$) then
we say that $x$ and $y$ are \textbf{matched} by $M$
(or that $y$ is the \textbf{partner} of $x$ in $M$);
if $M(x)=x$ then we say that $x$ is \textbf{unmatched} by $M$.

Given a matching $M$, define $d_M:V\to\R_+\cup\{\infty\}$ by
\begin{equation}\label{dMdef}
d_M(x)=
\ell\big(x,M(x)\big).
\end{equation}
The matching $M$ of $(V,E_\ell)$ is \textbf{stable} (with
respect to the function $\ell$)
if
\begin{equation}
\label{stablecondition}
\ell(x,y)\geq\min\left(d_M(x), d_M(y)\right)
\text{ for all $x$ and $y$.}
\end{equation}

We can interpret this definition as follows. Each point $x$ has
an order of preference among the other points; it prefers
to have a partner $y$ such that $\ell(x,y)$ is as
small as possible (but will remain unmatched rather than
being matched to another $y$ with $\ell(x,y)=\infty$).

\begin{prop}\label{prop:uniquestable}
Let $V$ be finite or countably infinite.
Suppose that the function $\ell$ satsfies the following conditions.
\begin{itemize}
\item[(i)]Distinct weights:
there are no $x,y,z\in V$ with $y\ne z$ such that $\ell(x,y)=\ell(x,z)<\infty$.
\item[(ii)]
Locally finite: for all $x\in V$ and all $r<\infty$,
the set $\{y\in V:\ell(x,y)<r\}$ is finite.
\item[(iii)]
No infinite descending paths:
there is no sequence of elements $x_0, x_1, x_2, \dots$ of $V$
such that $\ell(x_0, x_1)>\ell(x_1, x_2)>\ell(x_2, x_3)>\dots$.
\end{itemize}
Then there exists a unique stable matching $M$ of $(V,E_\ell)$.
If $x$ and $y$ are two points both left unmatched by $M$, then $\ell(x,y)=\infty$.
\end{prop}

Proposition \ref{prop:uniquestable} applies to the models in Theorems \ref{thm:asymmetricRd} and \ref{thm:symmetricRd}, as well as to the following more general setting.  Let $\mathcal{G}$ be an undirected graph with vertex set $\{1,\ldots,k\}$ with no parallel edges but possibly with self-loops.  An edge between $i$ and $j$ indicates that points of colours $i$ and $j$ are allowed to be matched to each other
(in which case we say that colours $i$ and $j$
are \textbf{compatible}). 
For the asymmetric model of Theorem \ref{thm:asymmetricRd}, $\mathcal{G}$ is one edge with a self-loop at one end; for Theorem \ref{thm:symmetricRd} it is a complete graph without self-loops.  
Given a probability vector $(p_1,\dots, p_k)$, 
consider a Poisson process on $\R^d$ of intensity 1;
let $V$ be the set of all the points of the process,
and let each $x\in V$ independently be given colour $i$
with probability $p_i$, for $1\leq i\leq k$. 
For two points $x,y\in V$ of respective colours $i$ and $j$ we let $\ell(x,y)=|x-y|$ if $i$ and $j$
are compatible, and $\ell(x,y)=\infty$ otherwise.


In the above setting, conditions (i) and (ii) in Proposition
\ref{prop:uniquestable} hold with probability 1
by basic properties of
the Poisson process. The fact that the Poisson process
has no infinite descending paths
with probability 1 is a special case
of Theorem 4.1 of Daley and Last \cite{DaleyLast},
so condition (iii) also holds. Hence indeed,
for the models of Theorems \ref{thm:asymmetricRd} and \ref{thm:symmetricRd}, with probability 1 there exists
a unique stable matching, and the same is true in the more general setting given in Theorem \ref{thm:meta}
at the end of Section \ref{sec:conclusionRd},

We will also apply Proposition \ref{prop:uniquestable} to the PWIT
in Section \ref{sec:PWIT} and to a variant of the
hierarchical metric on $\R$ in Section \ref{sec:conclusionhierarchical}.


We also make a useful observation about the
information necessary to determine whether or not
a given vertex $x$ is matched within some given distance $R$.
A \textbf{descending path} from $x$ with weights less than $R$
is a sequence $x_0, x_1, x_2, \dots, x_k$ with $x_0=x$
and
\begin{equation}\label{descendingdef}
R>\ell(x_0, x_1)>\ell(x_1, x_2)>\dots>\ell(x_{k-1}, x_k).
\end{equation}
Let $V^{\downarrow}_R(x)$ and $E^{\downarrow}_R(x)$ be the sets of
all vertices and respectively all edges
that are contained in any such path.

\begin{prop}\label{prop:whoismatched}
If conditions (i), (ii), (iii) of Proposition \ref{prop:uniquestable}
hold, then for all $x\in V$ and all $R>0$, the set $E^{\downarrow}_R(x)$ is finite. To determine whether $x$ is matched along an
edge of weight less than $R$ in the stable matching
(i.e.\ whether $d_M(x)<R$ where $M$ is the stable matching)
it suffices to know $E^{\downarrow}_R(x)$
and the collection of edge-weights
$\{\ell(y,z): \{y,z\}\in E^{\downarrow}_R(x)\}$.
\end{prop}

Finally, note that the definition of stable matching
only uses the relative ordering of edge-weights. If
the weights are rescaled
by applying the same strictly increasing
function to each finite weight, the set of stable
matchings does not change. Combining this with Proposition
\ref{prop:whoismatched}, we can in fact transfer information
about stable matchings from one graph to another, if
the local structure of the sets of descending paths
agrees in a suitable sense:
\begin{prop}\label{prop:stablerescaled}
Suppose that the set
$V$, with associated edge-weight function $\ell$, and
the set $\widetilde{V}$, with associated edge-weight function
$\widetilde{\ell}$, both satisfy conditions (i), (ii), (iii)
of Proposition \ref{prop:uniquestable}.
Let $M$ and $\widetilde{M}$ be the stable matchings of $V$
and $\widetilde{V}$ respectively.
Given $\widetilde{x}\in\widetilde{V}$ and $\widetilde{R}>0$, define
$\widetilde{V}_{\widetilde{R}}^{\downarrow}(\widetilde{x})\subseteq \widetilde{V}$
and $\widetilde{E}_{\widetilde{R}}^{\downarrow}(\widetilde{x})$
with respect to the function
$\widetilde{\ell}$ in the same way that $V_R^{\downarrow}(x)\subseteq V$
and
$E_R^{\downarrow}(x)$
were defined
with respect to the function $\ell$.

Let $f$ be a strictly increasing function $f:\R_+\cup\{\infty\}
\to\R_+\cup\{\infty\}$ such that $f(\infty)=\infty$.
Suppose there is a bijection
$\phi$ from $V_R^{\downarrow}(x)$ to $\widetilde{V}_{f(R)}^{\downarrow}(\widetilde{x})$
with $\phi(x)=\widetilde{x}$,
such that for each $u,v\in V_R^{\downarrow}(x)$,
$\{u,v\}\in E_R^{\downarrow}(x)$
iff $\{\phi(u),\phi(v)\}
\in \widetilde{E}_{\widetilde{R}}^{\downarrow}$,
and in that case
$\widetilde{\ell}(\phi(u), \phi(v))= f(\ell(u,v))$. Then
$\ell(x,M(x))<R$ iff
$\widetilde{\ell}(\widetilde{x},
\widetilde{M}({\widetilde{x}}))<f(R)$.
\end{prop}

We prove Propositions \ref{prop:uniquestable}, \ref{prop:whoismatched} and \ref{prop:stablerescaled}
in Section \ref{sec:stableproof}.
The proof is based on an inductive construction to identify
edges which must be included in any stable matching,
related to the approach used in \cite{HPPS}
for the special cases of one-type and symmetric two-type
matchings in $\R^d$ with weights given by Euclidean distance.

\pagebreak
\section{The Poisson-weighted infinite tree}
\label{sec:PWIT}
\subsection{Definition of the PWIT}
\label{subsec:PWITdef}
The \textbf{Poisson-weighted infinite tree}, or PWIT,
is an edge-weighted graph with vertex set
\[
\N^{\downarrow}=\bigcup_{k=0}^\infty \N^k=
\{\emptyset,1,2,\dots,11,12,\dots,21,22,
\dots,
111,112,\dots\},
\]
and edges $\{v,vj\}$ for each $v\in\N^{\downarrow}$ and $j\in\N$.
We say that $vj$ is a \textbf{child} of $v$.
For each $v\in\N^{\downarrow}$, let $(t^{(v)}_j: j=1,2,3,\dots)$
be the points of a Poisson process of rate 1 on $\R_+$
in increasing order,
and let these processes be independent for different $v$.
Then let $t^{(v)}_j$ be the weight associated to the edge
$\{v,vj\}$
(which we will also sometimes write as $t(v,vj)$).

The PWIT was introduced by Aldous and Steele \cite{AldousSteele}
and often arises in applications as a scaling
limit of the complete graph with edges weighted
by i.i.d.\ random variables.
We will explain how it also gives a scaling limit
of a Poisson process in high-dimensional Euclidean space.
Stable matchings on the PWIT can be analysed quite precisely,
and we will be able to use them to study the behaviour of stable matchings
in $\R^d$ for large $d$.

First we mention briefly the way in which the PWIT
arises as a limit of the weighted complete graph.
This can be formalised in many different ways;
in particular the framework
of \textit{local weak convergence} is often
used (see for example \cite{AldousSteele}),
but the following less technical approach gives
the essential idea. Consider the complete graph $K_n$
with i.i.d.\ weights attached to the edges
which are, say, exponential with rate $1/n$.
Fix some vertex $v$ and some ``radius" $R>0$,
and consider the subgraph of $K_n$ created by
the collection of all paths from $v$ which have total weight at most $R$.
Similarly we can consider the subtree of the PWIT
created by the collection of all paths from
the root which have total weight at most $R$.
Then for any given $R$, we can couple
the complete graph with the PWIT so that, with
probability tending to 1 as $n\to\infty$, there is an
isomorphism between these two subgraphs which
identifies $v$ with the root of the PWIT, and which preserves the edge weights.

Now we motivate informally the idea of the PWIT
as a limit of the Poisson process in $\R^d$ as $d\to\infty$.
Let
\[
\omega_d=\frac{\pi^{d/2}}{\Gamma\left(1+\frac{d}{2}\right)}.
\]
Then the volume of a ball of radius $r$ in $\R^d$ is $\omega_d r^d$.

Consider a Poisson process of rate $1$ in $\R^d$, as seen from a ``typical point", located at the origin and denoted by $O$. (We will make this notion precise by considering the Palm version of the Poisson process in Section
\ref{sec:coupling}.) The point $O$ will correspond to the root
of the PWIT. Let $x_1, x_2, x_3,\dots$ be the other points of the process,
written in order of their distance from $O$.
Then the sequence $\omega_d|x_1|^d, \omega_d|x_2|^d, \omega_d|x_3|^d,\dots$
forms a Poisson process of rate 1 on $\R_+$. Rescaled in this
way, these distances correspond to the weights
$t^{\emptyset}_1,
t^{\emptyset}_2,
t^{\emptyset}_3,\dots$
on the edges connecting the root of the PWIT to its children.

Note that $\omega_d^{1/d}|x_1|$ converges in probability to 1 as $d\to\infty$. In fact, for any $\epsilon>0$, the probability that there exists a point
of the process within distance $\omega_d^{-1/d}(1-\epsilon)$ of $O$ decays
exponentially with $d$, while the expected number of points within
distance $\omega_d^{-1/d}(1+\epsilon)$ increases exponentially.

Now consider in turn the points closest to $x_1$. Other than the origin,
let these points be $x_{1,1}, x_{1,2}, x_{1,3},\dots$
in order of distance from $x_1$. Similarly rescaled,
their distances from $x_1$ again converge to a Poisson process,
and $\omega_d^{1/d}|x_{1,1}-x_1|$ converges in probability to 1 as $d\to\infty$.
On the other hand, for large $d$, we expect $x_1$ and $x_{1,1}-x_1$
to be approximately orthogonal, so that $x_{1,1}$ is at distance
approximately $\sqrt{2}\omega_d^{-1/d}$ from $O$. In particular, $x_{1,1}$
is not among the nearest neighbours of $O$
(we expect to find exponentially many closer points).

We can extend by considering paths from the origin consisting of distinct points $x_0=O, x_1, x_2, \dots, x_k$, in which each $x_j$ is one of
the $m$ nearest neighbours of $x_{j-1}$. For given $k$
and $m$, with high probability
as $d\to\infty$ no such path ends in a point $x_k$ which is one of the $m$ nearest neighbours of $O$. This explains why the acyclic structure
of the PWIT gives an appropriate limit for the graph of
near neighbours in the Poisson process on $\R^d$.

In this way we could give a result similar to that mentioned for the complete graph above, comparing the structure of the PWIT and the Poisson process restricted to paths of total (rescaled) weight $R$, corresponding to local weak convergence. (This mode of convergence would be sufficient, for example, to obtain convergence of
the minimal spanning tree on the points of a Poisson process in a finite box of $\R^d$ to the minimal spanning forest of the PWIT -- see Theorem 5.4 of \cite{AldousSteele}
for a general result concerning covergence
of the minimal spanning tree under local weak convergence.) However, to analyse stable matchings we need a different mode of convergence, concerning the subgraph obtained
by taking all descending paths from the
root (in the PWIT) or the origin (in $\R^d$)
with weights (or distances) less than $R$;
in the same way as at (\ref{descendingdef}), a descending
path is a path such that the successive edge
weights (or distances) form a decreasing sequence.
We show that with high probability as $d\to\infty$,
the collection of descending paths in the two models
can be coupled so that (after rescaling of distance)
their graph structure is identical in the sense of
Proposition \ref{prop:stablerescaled}.
This will allow us to approximate certain intensities
in the stable matching model in $\R^d$ (for example, the intensity of points of a given type which are not matched by the stable matching) by probabilities involving the matching of the root of the PWIT.

Note that while the edge weights of the PWIT correspond to rescaled distances, we do not think of these weights as defining a graph
distance, or indeed giving any metric. The weight
(corresponding to rescaled distance)
between $\emptyset$ and its child $1$ in the PWIT is $t_1^{\emptyset}
\sim\textrm{Exp}(1)$,
and that between $1$ and its child $11$ is $t_1^1\sim\textrm{Exp}(1)$,
but, asymptotically as $d\to\infty$, the rescaled distance between
$\emptyset$ and its ``grandchild" $11$ goes to $\infty$.

\subsection{Stable matchings on the PWIT}
\label{sec:PWITmodel}
In the rest of the section
we analyse stable matching problems where
the set of points is given by the vertices of the PWIT.

Given a probability vector $(p_1, \dots, p_k)$,
let each vertex of the PWIT have type (or colour) $i$ with
probability $p_i$, independently for different vertices.
(Note then that for each $v$, the weights of edges from
$v$ to its children of colour $i$ form a Poisson process
of rate $p_i$, indepdendently for different $i=1,\dots, k$
and $v\in\N^{\downarrow}$.) As in the spatial models already considered,
the colours determine which pairs of points are allowed to be matched.  To apply Proposition \ref{prop:uniquestable} to the PWIT, let $\mathcal{G}$ be a graph of colour compatibilities as discussed earlier, let $V=\N^\downarrow$ be the set of all vertices of the PWIT, and for any
$v\in \N^\downarrow$ and $j\in\N$ let $\ell(v,vj)=t^{(v)}_j$ if the colors of $v$ and $vj$ are compatible; for all other pairs of vertices $u,v\in\N^\downarrow$ (i.e.\ for incompatible colors or non-neighboring vertices) let $\ell(u,v)=\infty$.  Note in particular that we do not allow non-neighbours in the PWIT to be matched to each other.

From elementary properties of the Poisson process,
with probability 1, all the weights in the PWIT are
distinct, and any vertex has only finitely many edges
with weights falling in any given compact interval.
This gives conditions (i) and (ii) of Proposition
\ref{prop:uniquestable}, while condition (iii)
on the absence of infinite descending chains
will be given by Lemma \ref{lemma:treesize} below.
So in all the cases of interest, with probability 1 there
exists a unique stable matching.

The PWIT has a recursive structure.
The subtrees rooted at each child $1,2,\dots$ of the root
$\emptyset$ have the structure of independent copies of the PWIT.
We can consider the stable matching problem on any of these subtrees.
If $j$ is a child of the root, along an edge with weight $t_j^{(\emptyset)}$,
say that $j$ is ``available" (to the root) if $j$ is not matched
along an edge with weight
less than $t_j^{(\emptyset)}$ in the stable matching of
the subtree rooted at $j$.
Then the root is matched to the nearest of its children
which is both available and has a compatible colour
(and is unmatched if no such child exists).

\subsubsection{One-type matching}
As an introduction we start with the simplest case, where there is only one type of point (so $p_1=1$) and any pair $\{v, vj\}$ with $v\in \N^{\downarrow}$, $j\in \N$ may be matched. That is, $\ell$ is the symmetric function given by
\begin{equation}\label{elldefPWIT}
\ell(v,vj)=t^{(v)}_j,
\end{equation}
and $\ell(v,w)=\infty$ whenever $v$ and $w$ are
not joined by an edge in the PWIT.

For $t\geq 0$, let $x(t)$ be the probability that the root is not matched along an edge with weight less than $t$. By the recursive structure
of the PWIT, conditional on the weights
$t_1^{(\emptyset)}=t_1$, $t_2^{(\emptyset)}=t_2,\dots$
from the root to its children,
these children are available independently with
probabilities $x(t_1), x(t_2), \dots$.
In fact, the process of available children of the root
forms a inhomogeneous Poisson process with rate $x(t)$
on $\R_+$.
The root is matched to the first point of this process.
So $x(t)$ is given by the probability that this process has no points
in $[0,t)$, giving
\[
x(t)=\exp\left\{-\int_0^t x(u) du\right\}.
\]
Hence we have
\begin{equation}
\label{one-type-equation}
x'(t)=-x(t)^2.
\end{equation}
Using $x(0)=1$ gives the exact solution
$x(t)=1/(t+1)$. Since $x(t)\to0$ as $t\to\infty$,
the root is matched with probability 1. (However,
since $\int_0^{\infty}x(t)dt=\infty$,
the weight of the edge along which it is matched has infinite mean.)

\subsubsection{Asymmetric two-type matching}
\label{subsubsec:asymmPWIT}
Now we study the asymmetric two-type model
corresponding to the one studied in Theorem
\ref{thm:asymmetricRd} for points in $\R^d$.
Let each vertex of the PWIT independently be red with probability
$1-\epsilon$ and blue with probability $\epsilon$.
Red-red and red-blue
matches are allowed, but blue-blue matches are not.
That is, the function $\ell$ is defined as at (\ref{elldefPWIT})
except that now $\ell(v,vj)=\infty$ if $v$ and $vj$ are both blue.

Let $r(t)$ be the probability that the root is red
and not matched along an edge of weight less than $t$, and $b(t)$
the probability that the root is blue and not matched
along an edge of weight less than $t$.
By an analogous argument to the one-type case above,
the processes of available red and blue children of the root
form independent inhomogeneous Poisson processes
of rates $r(t)$ and $b(t)$ on $\R_+$.

We have
\begin{gather*}
r(t)=(1-\epsilon)\exp\left\{-\int_0^t \big[r(u)+b(u)\big]du\right\}\\
b(t)=\epsilon\exp\left\{-\int_0^t r(u) du\right\},
\end{gather*}
so that
\begin{align}
\nonumber
r'(t)&=-r(t)\big(r(t)+b(t)\big)\\
\label{rb-equation}
b'(t)&=-b(t)r(t),
\end{align}
with initial conditions $r(0)=1-\epsilon$
and $b(0)=\epsilon$.
Certainly $r(t)\to 0$ as $t\to\infty$ (for example we have $r'(t)<-r(t)^2$, so by comparison
to (\ref{one-type-equation}), $r(t)<1/(t+1)$).
We can ask whether or not there are unmatched blue points;
that is, does $b(t)$ also converge to 0 as $t\to\infty$?

\begin{figure}[t]
\begin{center}
\includegraphics[width=0.8\textwidth]{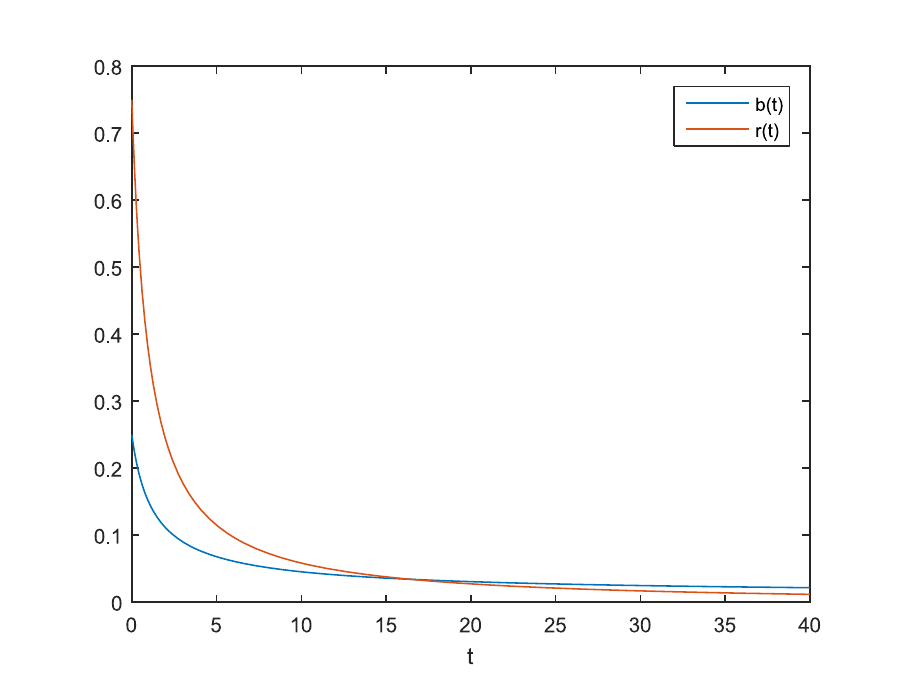}
\caption{
\label{fig:rb}
Numerical solution of the system (\ref{rb-equation})
describing the evolution of the asymmetric two-type system on the PWIT,
with $r(0)=0.75$ and $b(0)=0.25$.
As $t\to\infty$, $r(t)$ decays exponentially to 0
and $b(t)$ converges to $b(\infty)\approx 0.0124$.
}
\end{center}
\end{figure}

Write $R(t)=-\log r(t)$ and $B(t)=-\log b(t)$.
Then $R$ and $B$ are increasing with $t$,
and $R(t)\to\infty$ as $t\to\infty$.
We can derive $B$ as a function of $R$, and ask whether
$B\to\infty$ as $R\to\infty$.

We have
\begin{align*}
R'&=-\frac{r'}{r}=r+b=e^{-R}+e^{-B}\\
B'&=-\frac{b'}{b}=r=e^{-R},
\end{align*}
which gives
\[
\frac{dR}{dB}=\frac{e^{-R}+e^{-B}}{e^{-R}}=1+e^{R-B},
\]
leading to $\frac{d(R-B)}{dB}=e^{R-B}$
which has general solution
\[
R-B=-\log(-B+c).
\]
We see that as $B\uparrow c$, $R\to\infty$.
Hence for the original system $B(t)\to c$ as $t\to\infty$.

To find $c$, we use
$r(0)=1-\epsilon$ and $b(0)=\epsilon$,
so that $R(0)=-\log(1-\epsilon)$ and $B(0)=-\log \epsilon$; this gives
gives $c=\frac{1-\epsilon}{\epsilon}+\log\frac1\epsilon$.

Then $b(t)\to b(\infty)=e^{-B(\infty)}=e^{-c}=\epsilon e^{-1/\epsilon+1}
=e^{-1/\epsilon+1}b(0)$.

So we see that a proportion $e^{-1/\epsilon+1}$ of the blue
points remain unmatched (or, more formally, this is the
conditional probability
that the root remains unmatched, given that it is blue).

\subsubsection{Symmetric multi-type matching}
\label{subsubsec:symm}
Now we turn to the model corresponding to the setting of
Theorem \ref{thm:symmetricRd}, in which there are $k$ types
with probabilities $p_1, \dots, p_k$, and two points may be matched
if their types are different.
Now we define $\ell$ as at (\ref{elldefPWIT}) unless
$v$ and $vj$ have the same colour, in which case $\ell(v,vj)$ is infinite.

Let $x_i(t)$ be the probability that the root itself has type $i$
and is not matched along an edge of weight less than $t$.
As before, the weights of edges from the root leading to available children
of type $i$ form inhomogeneous Poisson processes
of rates $x_i(t)$, independently for $i=1,2,\dots,k$.

Then $x_i(0)=p_i$ and
\begin{equation}\label{symmetric-equation}
x'_i(t)=-x_i(t)\sum_{j\ne i}x_j(t).
\end{equation}
Writing $X_i(t)=-\log x_i(t)$ this gives
\begin{align}\label{deriv}
X_i'(t)&=\frac{-x_i'(t)}{x_i(t)}\\
\nonumber
&=\sum_{j\ne i}x_j(t),
\end{align}
and so
\begin{align*}
X_i'(t)-X_j'(t)&=x_j(t)-x_i(t)\\
&=e^{-X_j(t)}-e^{-X_i(t)}.
\end{align*}
Then if $x_j(t)<x_i(t)$,
or equivalently $X_i(t)<X_j(t)$, then the derivative
of $X_j(t)-X_i(t)$ is strictly positive. Hence
in particular if $p_j<p_i$ then $x_j(t)\leq x_i(t)$ for all $t$.

Suppose the maximum initial density is
attained by at least two types; say $p_1=p_2\geq p_j$ for all $j$.
Then by symmetry $x_1(t)=x_2(t)$ for all
$t$. It's impossible for a positive proportion of points
of two different types to remain unmatched
(in particular, the root would have an unmatched child
of a different colour with probability 1, and this would contradict
the final statement of Proposition \ref{prop:uniquestable}).
Hence in this case all points are matched.

Suppose on the other hand that
there is a unique type with highest initital probability.
Then we will show that a positive proportion of
points of this type remain unmatched:
\begin{prop}
If $p_1>p_j$ for all $j>1$, then
$\lim_{t\to\infty} x_1(t)>0$.
\end{prop}

\begin{proof}
Without loss of generality, assume that $p_1>p_2\geq p_3\geq \dots\geq p_k$.
As noted above, then also $x_2(t)\geq x_j(t)$ for all $t$ and all
$3\leq j\leq k$.
Since only the type with maximum initial probability can have points left
unmatched, we know that $x_2(t)\to 0$, i.e.\ $X_2(t)\to\infty$.
Now write $Z(t)=X_2(t)-X_1(t)$.
From (\ref{deriv}), we get
\begin{align*}
\frac{dZ(t)}{dX_2(t)}
&=\frac{d(X_2(t)-X_{1}(t))}{dX_{2}(t)}\\
&=\frac{x_1(t)-x_2(t)}
{\sum_{j>1}x_j(t)}\\
&=\frac{\exp(-X_1(t))-\exp(-X_{2}(t))}
{\exp(-X_1(t))+\sum_{j\geq 3}\exp(-X_j(t))}\\
&\geq
\frac{\exp(-X_1(t))-\exp(-X_2(t))}
{\exp(-X_1(t))+(k-2)\exp(-X_2(t))}\\
&=\frac{\exp(-X_{2}(t))(\exp(Z(t))-1)}
{\exp(-X_{2}(t))(\exp(Z(t))+(k-2))}\\
&=\frac{\exp(Z(t))-1}{\exp(Z(t))+(k-2)}\\
&=1-\frac{k-1}{\exp(Z(t))+(k-2)}.
\end{align*}

This derivative is always positive since
$Z(t)>0$ for all $t$; hence in fact $Z(t)$
is increasing as a function of $X_{2}(t)$,
and this derivative is bounded away from 0.
So $Z(t)\to\infty$ as $X_{2}(t)\to\infty$,
i.e.\ as $t\to\infty$.

This gives that $X_{2}(t)-X_1(t)\to\infty$,
i.e. that $x_{2}(t)/x_1(t)\to0$.
Then also $x_j(t)/x_1(t)\to0$ for all $j>1$.

So for some $t$,
\begin{equation}\label{moretype1}
x_1(t)>x_2(t)+\dots+x_k(t).
\end{equation}
Now the intuition is that since, looking at points unmatched
within weight $t$,
the density of type-1 points
is higher than the density of all other types put together,
it is impossible to match all the type-1 points.
To see this directly, one can use (\ref{deriv})
to observe that the derivative of $x_1(t)-\sum_{j\geq 2}x_j(t)$
is always non-negative (heurisitically,
this reflects the fact that each match involves at most one type-1 point and
at least one point of another type); combining with (\ref{moretype1})
gives that $x_1(t)$ stays bounded away from 0 as $t\to\infty$.
\end{proof}

\begin{remark}
In the case where $p_i, 1\leq i\leq k$ take only two
distinct values, we can solve exactly.
Consider for example the case where
$p_1>p_2=\dots=p_k$, which (up to reordering) is the only such case
where some points will remain unmatched.

Then by symmetry between all coordinates except the first,
\begin{align*}
x_1'(t)&=-(k-1)x_1(t)x_2(t)\\
x_2'(t)&=-(x_1(t) x_2(t) +(k-2)x_2(t)^2).
\end{align*}
We get
\[
\frac{dx_2}{dx_1}=\frac{1}{k-1}+\frac{k-2}{k-1}\frac{x_2}{x_1},
\]
which is solved by
\[
x_2(t)=x_1(t)-cx_1(t)^{\frac{k-2}{k-1}}.
\]
From the initial values $x_1(0)=p_1$, $x_2(0)=p_2$
we obtain $c=(p_1-p_2)p_1^{-\frac{k-2}{k-1}}$.
Then considering $t\to\infty$ and using $x_2(\infty)=0$,
we have
\begin{align*}
x_1(\infty)&=(p_1-p_2)^{k-1} p_1^{-(k-2)}\\
&=\left(1-\frac{p_2}{p_1}\right)^{k-1} x_1(0).
\end{align*}
We can interpret the quantity $x_1(\infty)/x_1(0)$
as the ``proportion of points of type 1 left unmatched"
(more precisely, the probability that the root is unmatched,
given that it has type 1).
Looking for asymptotics as the difference between the
initial probabilities becomes small,
we can put for example
$p_1=\frac1k+(k-1)\delta$
and $p_2=\frac1k-\delta$. Then  we obtain
\begin{align*}
\frac{x_1(\infty)}{x_1(0)}
&=\left(\frac{k^2\delta}{1+k(k-1)\delta}\right)^{k-1}
\\
&\sim \,\,\,k^{2(k-1)}\delta^{k-1}
 \text{ as } \delta\downarrow 0.
\end{align*}
\end{remark}

\section{Coupling the PWIT and a Poisson process in $\R^d$}
\label{sec:coupling}
\subsection{Palm version}
\label{subsec:Palm}
Consider a simple point process in $\R^d$ with finite intensity.
The \textbf{Palm version} of the process is obtained, informally speaking, by conditioning on the presence of a point at the origin. One can also describe the Palm version as giving the distribution of the process ``as seen from
a typical point". This notion can be formalised in various equivalent
ways. For example, let $\Pi$ be the point process, let
$[\Pi]$ denote the set of its points, and, for $y\in \R^d$,
let $\theta^y(\Pi)$ denote the process obtained by translating
$\Pi$ by $y$; then the probability of an event $A$
for the Palm version $\Pi^{\downarrow}$ of $\Pi$ can be defined
by
\[
\P(\Pi^{\downarrow}\in A)=
\frac{
\E\#\left\{x\in[\Pi]\cap[0,1]^d : \theta^{-x}(\Pi)\in A\right\}.
}
{
\E\#\left\{x\in[\Pi]\cap[0,1]^d\right\}
}.
\]
In the case of a Poisson process, the Palm version
has a particularly straightforward description; it can be obtained
simply by adding a point at the origin to a configuration
drawn from the original measure. See for example
Chapter 11 of Kallenberg \cite{Kallenberg} for extensive details.

Our multi-type models
add information about the colours of the points of the Poisson process.
In the language of \cite{Kallenberg}, this information
can be taken as a \textit{stationary background}. To obtain the Palm
version, we add a point at the origin whose colour is again
drawn according to the same distribution as the other points
(and independently of the rest of the configuration).

Intensities of various types of point in the original process
can then be related to probabilities involving the point at the origin in
the Palm version. In particular, the intensity of
points of type $i$ which are unmatched in the stable matching
is given by the probability that the point at the origin in the Palm
version has type $i$ and is unmatched in the stable matching.

\subsection{Descending paths}
In the PWIT, let a
\textit{descending path from the root with weights less than $T$}
be a sequence
$v_0, v_1, \dots, v_k$ of points of $\N^{\downarrow}$,
where $v_0$ is the root, where
$v_i$ is a child of $v_{i-1}$ for $i=1,2,\dots, k$, and where
\[
T>t(v_0, v_1)>t(v_1,v_2)>\dots>t(v_{k-1}, v_k),
\]
where $t(v_{i-1}, v_i)$ is the weight of the edge between
$v_{i-1}$ and $v_i$.

For the Palm version of the
Poisson process in $\R^d$,
let a \textit{descending path from the origin with distances less than $R$}
be a sequence of distinct points of the process $x_0, x_1, \dots, x_k$ where $x_0$ is the origin and
where
\[
R>|x_0-x_1|>|x_1-x_2|>\dots>|x_{k-1}-x_k|.
\]

For finite $T$, with probability 1 the set of descending paths from
the root within distance $T$ in the PWIT is finite and contains only
finite paths (see Lemma \ref{lemma:treesize}).
(The analogous property is also true for the
Palm version of the Poisson process in $\R^d$;
this is a special case of Theorem 4.1 of
Daley and Last \cite{DaleyLast}.)

From Proposition \ref{prop:whoismatched},
we know that for the stable matching on the PWIT,
the event that the point at the origin is matched
to a child along an edge with weight less than $T$
is in the sigma-algebra generated by
the graph of descending paths from the origin
within distance $T$ (including the information
about the colours of points); similarly in $\R^d$ for the event
that the origin is matched to a point at distance less than $R$.

\subsection{Description of the coupling}
Throughout this section, we consider the Palm version
of the Poisson process in $\R^d$.
Suppose $T$ and $R$ are related via $T=f_d(R):=\omega_d R^d$
so that a ball of radius $R$ in $\R^d$ has volume $T$.

We aim to couple the collection of descending paths from the root
with weights less than $T$ in the PWIT with
the collection of descending paths from the origin
with distances less than $R$ in $\R^d$,
in such a way that their graph structure is identical,
and such that the weight $t$ of an edge in the PWIT
and the distance $r$ between the corresponding points in $\R^d$
are related by $t=f_d(r)=\omega_d r^d$.
Specifically, we want to arrange that
the hypotheses of Proposition
\ref{prop:stablerescaled} are satsfied
with high probability.
As in Section \ref{sec:uniqueness},
let $V_R^\downarrow(O)$ and $E_R^\downarrow(O)$
be the sets of points, and respectively edges,
contained in some descending path from the origin
with distances less than $R$ in $\R^d$,
and let $\widetilde{V}_T^\downarrow(\emptyset)$
and $\widetilde{E}_T^\downarrow(\emptyset)$
be the set of points,
and respectively edges,
contained in some descending path
from the root in the PWIT with weights less than $T$.

\begin{prop}\label{prop:couplingsuccess}
There exist absolute constants $\alpha>0$ and $c$ such that the following holds.
Let $T\geq 1$, let $d>cT$, and let $R=f_d^{-1}(T)$.

Then we can couple the PWIT model
with the Palm version of the Poisson model
in $\R^d$ such that with probability at least
$1-e^{-\alpha T}$,
there exists a bijective map
$\phi:
V_R^\downarrow(O) \mapsto
\widetilde{V}_T^\downarrow(\emptyset)$
with the following properties:
\begin{itemize}
\item[(i)]
$v_i$ and $\phi(v_i)$ have the same colour for all $i$;
\item[(ii)]
$v_0, v_1, \dots, v_k$ is a descending
path from the origin with weights less than
$R$ in $\R^d$ if and only if $\phi(x_0), \phi(x_1), \dots ,\phi(x_k)$
is a descending path from the root with weights less than
$T$ in the PWIT, and if so then
$t(\phi(v_{i-1}), \phi(v_i))=
f_d(|x_i-x_{i=1}|)$
for $i=1,2,\dots,k$.
\end{itemize}
\end{prop}

Using this result we can apply
Proposition \ref{prop:stablerescaled} to obtain
that (with probability at least $1-e^{-\alpha T}$)
the origin in $\R^d$ is matched within distance $R$ if and only if the root of the PWIT is matched within distance $T$.

Our strategy is to couple a procedure that
explores the
collection of descending paths in $\R^d$ with one
that generates the collection for the PWIT,
aiming to maintain the bijection as described above.
If certain events occur for the Poisson configuration in $\R^d$,
the coupling will fail (and we terminate the procedure -- on this set we couple
the two processes in an arbitrary way so as to
maintain the required marginals);
but if the procedure reaches the end then it is guaranteed that
a bijection as described above exists. We will give a lower bound
for the probability that the coupling reaches the end successfully.

First we describe the procedure to explore the collection in $\R^d$.
We will abandon this
exploration if it ever discovers a point in
$\R^d$ that can be arrived at via two different descending paths
from the origin within distance $R$
(if this happens, it is certainly impossible to couple
successfully, since in the case of the PWIT the
collection of descending paths has a tree structure).

For a point $x\in V_R^\downarrow(O)$, other than the origin,
we say the \textit{parent} of $x$ is the point that precedes it
in the descending path from the origin to $x$ with
distance less than $R$. This is unambiguously defined for as long as the procedure keeps running, since if more than one such path is ever discovered, the procedure stops.

In fact, we will be more conservative. If we find two points of $V_R^\downarrow(O)$
which are closer than $R$ to each other, and neither is the parent of the other, we will abandon the procedure. (Note that this must occur if there is any point that can be arrived at via two different descending paths from the origin within distance $R$.) Furthermore, we will also abandon the procedure if it ever finds a parent and child which are closer than
$R/2$ to each other.

We explore space gradually, discovering points of $V_R^\downarrow(O)$
as we proceed. We maintain an ordered list
of points which we have discovered, say $x_0, x_1, \dots, x_k$,
where $x_0$ is the origin.
Let $r_0=R$ and for $j>0$, let $r_j$ be the distance to $x_j$ from its parent, which is $x_i$ for some $i<j$; note that $r_j<r_i$ by the descending path property.

We ``process" the points in order; to process $x_j$, we look for new points in the open ball $B_j:=B(x_j, r_j)$, and add any such new points to the end of the list. These are the points whose parent is $x_j$, i.e.\ the points which can follow $x_j$ in a descending path.

Suppose our list is $x_0, \dots, x_k$,
and we are currently processing point $x_j$ where $j\leq k$.
This means that
we have already processed $x_0, \dots, x_{j-1}$, and so
the region $A_j$ defined by
\begin{equation}\label{Ajdef}
A_j:=B_j\cap\bigcup_{0\leq i<j} B_i
\end{equation}
has already been explored,
and is known to contain no Poisson points
other than $x_j$ (if there had been any such
point, the procedure would have terminated
at an earlier stage since
that point and $x_j$
would have been too close to each other.)



Now we describe how to couple this exploration procedure
with a process which generates
the tree of descending paths with weights less than $T$
in the PWIT. First a useful observation:
\begin{lemma}\label{lemma:simplescaling}
Let $r>0$ and $t=f_d(r)$.
Let $x\in\R^d$ and let $x_1, x_2, \dots, x_k$
be the points of a Poisson process of rate 1 in $B(x,r)$.
Let $t_i=f_d(|x_i-x|)$. Then $t_1, t_2, \dots, t_k$
are the points of a Poisson process of rate 1 on $[0,t]$.
\end{lemma}
\begin{proof}
This is immediate from basic properties of the Poisson process,
and the fact that the ball
$\{y:f_d(|y-x|) <s\}=B(x,f_d^{-1}(s))$
has volume $s$.
\end{proof}

We start off with $x_0$ the origin in $\R^d$ and $v_0$ the
root of the PWIT. For as long as the coupling is successful,
at each stage we have a set of points $v_0, \dots, v_k$
which have been discovered in the PWIT, and which correspond
to the points $x_0, \dots, x_k$ discovered in $R^d$. The root
of the PWIT is $v_0$ and corresponds to the origin in $\R^d$,
which is $x_0$.
If $v_i$ is the parent of
$v_j$ in the PWIT, let $t_j$ be the weight of the edge
between $v_i$ and $v_j$;
then $x_i$ is the parent of $x_j$ in the sense described
earlier for $\R^d$, and $t_j=f_d(|x_i-x_j|)$.

At the same time as processing $x_j$ in $\R^d$, we process $v_j$ in the PWIT. Processing $v_j$ involves generating
the children of $v_j$ which are connected to $v_j$
along edges with weights in $[0,t_j)$.
Notice that $t_j$ is precisely the volume of $B(x_j, r_j)$.

Hence we can couple the children of $y_j$ in the interval $[0,t_j)$
with a Poisson process of rate 1 in $B_j=B(x_j, r_j)$ in
such a way that the weights on the edges from $y_j$ and
the distances of the points from $x_j$ are related according
to the scaling in Lemma \ref{lemma:simplescaling}.

Notice that at this stage of the exploration procedure,
the new points we discover are not a Poisson process on the
whole of $B_j$; as observed above at (\ref{Ajdef}),
the subset $A_j$ of $B_j$ has already been explored. Hence
to generate a set of children and edge weights according
to the correct distribution, we supplement the new
points in $B_j\setminus A_j$ (which are independent of everything
seen in the procedure so far, since this region
has not yet been explored so far) with an extra Poisson process
of rate 1 in $A_j$, again chosen independently of the
points in $B_j$ and of everything else seen so far.
In this way we obtain a Poisson process of rate 1 in $B_j$,
which is independent of the previous history of the procedure,
and we use the correspondence in Lemma \ref{lemma:simplescaling}
to derive the weights to children of $y_j$ which lie in $[0,t_j)$.

If in fact the extra Poisson process in $A_j$ contains at least one point,
we are in trouble, because we cannot maintain the correspondence between the new points found in $\R^d$ and the new vertices added
to the PWIT. In this case we abandon the procedure.
However, if this supplementary process in $A_j$ is empty,
then we can maintain the bijection and the procedure continues.

If the procedure finishes (i.e.\ runs out of new points in $\R^d$
to process) without abandoning, then it provides a bijection between
$V_R^\downarrow(O)$ and $\widetilde{V}_T^\downarrow(\emptyset)$ as required for Proposition \ref{prop:couplingsuccess}.

We summarise the ways that the procedure may fail at step $j$, i.e.\
at the step where we process the point $x_j$:
\begin{itemize}
\item[(1)] Within $B_j\setminus A_j$, we find a child of $x_j$
which is within distance $R/2$ of $x_j$.
\item[(2)] Within $B_j\setminus A_j$, we find two children of $x_j$
which are within distance $R$ of each other.
\item[(3)] Within $B_j\setminus A_j$, we find a child of $x_j$
which is within distance $R$ of a previously discovered point.
\item[(4)] The supplementary Poisson process of rate 1 on $A_j$
contains one or more points.
\end{itemize}
For $m=1,2,3,4$, let us write $\cE_j^{(m)}$ for the event
that the procedure successfully completes steps $1,\dots, j-1$,
and then failure type $(m)$ above occurs at step $j$.
(Under this definition it is possible that $\cE_j^{(m)}$ and $\cE_j^{(m')}$
both occur for different $m$ and $m'$, but it is not possible that
$\cE_j^{(m)}$ and $\cE_{j'}^{(m')}$ both occur for
$m,m'\in\{1,2,3,4\}$ and for different $j$ and $j'$.)
In the next section we bound the probabilities of each of these types of failure.

If the procedure does fail at step $j$, we do not proceed to step $j+1$.
For the sake of being specific about the coupling, let us say that we
we generate the rest of the subtree of the PWIT spanned by $\widetilde{V}_T^\downarrow(\emptyset)$
according to its distribution conditional on the part of the structure
already created at steps $1,\dots, j$, and independently of any further
information about the process in $\R^d$.

\subsection{Bounding the probability of failure of the coupling}
As above, throughout this section we set $T=\omega_d R^d$, the volume of a ball
of radius $R$ in $\R^d$.

\label{boundsection}
\begin{lemma}\label{lemma:E1bound}
For all $j$, $\P(\cE_j^{(1)})\leq \left(\frac12\right)^d T$.
\end{lemma}
\begin{proof}
$\cE_j^{(1)}$ is the probability that the procedure
reaches step $j$, and then we find at least one new point
in $B(x_j, R/2)\setminus A_j$. Since (independently of everything
seen so far) the points in that set form a Poisson process of rate 1,
this probability is bounded above by the volume of $B(x_j, R/2)$,
which is $(1/2)^d T$ as required.
\end{proof}

\begin{lemma}
If $y,z\in \R^d$ with $|y-z|\geq R/2$,
then
\begin{equation}\label{volbound}
\vol\big(B(y,R)\cap B(z,R)\big)
\leq \left(\frac{15}{16}\right)^{d/2} T.
\end{equation}
\end{lemma}
\begin{proof}
Any point $x$ in the intersection of the two balls
is at distance at most $h$ from the midpoint of $y$ and $z$,
where $h^2=R^2-(|y-z|/2)^2\leq (15/16)R^2$.
(This can be easily checked by considering the plane which contains $x$, $y$ and $z$.)
Hence the intersection is contained in a ball of radius $h$,
whose volume is $(h/R)^d T=(15/16)^{d/2} T$.
\end{proof}

\begin{lemma}\label{lemma:E2bound}
For all $j$,
$\P(\cE_j^{(2)}\setminus \cE_j^{(1)})\leq \left(\frac{15}{16}\right)^{d/2}T^2$.
\end{lemma}
\begin{proof}
If $\cE_j^{(2)}$ happens but $\cE_j^{(1)}$ does not, then
at step $j$ we find a pair of new points, say $y$ and $y'$,
which are both between distance $R/2$ and $R$ from $x_j$, and are within distance
$R$ of each other.

The expected number of such pairs is no more than
\[
\int_y I\big(y\in B(x_j, R)\setminus B(x_j, R/2)\big)
\int_{y'} I\big(y'\in B(x_j, R)\cap B(y,R) \big)dy' dy.
\]
Using (\ref{volbound}), this is bounded above by
\[
\big[\vol B(x_j, R)-\vol B(x_j, R/2)
\big]
\left(\frac{15}{16}\right)^{d/2} T
\]
which is less than $(15/16)^{d/2}T^2$ as desired.
\end{proof}

\begin{lemma}\label{lemma:E3bound}
For all $j$ and $K$,
$\P(\cE_j^{(3)} \cap \{|\widetilde{V}_T^\downarrow(\emptyset)|\leq K\})
\leq (K-1)\left(\frac{15}{16}\right)^{d/2}T$.
\end{lemma}
\begin{proof}
If the procedure runs successfully to step $j$,
and $|\widetilde{V}_T^\downarrow(\emptyset)|\leq K$, then at step $j$ (when we come to process the point $x_j$),
the set of already discovered points is $x_1,\dots, x_k$ for some
$k$ with $j\leq k\leq K$.

We want to bound the probability that we then find a new point inside $B_j$
which is within distance $R$ of some $x_i$, $i\ne j$, $i\leq k$.

This is at most
\[
\vol
\bigcup_{i\leq k, i\ne j} \big(B(x_i, R)\cap B(x_j, R)\big).
\]
But if indeed the procedure has been successful so far, then
in particular $|x_j-x_i|\geq R/2$ for all such $i$.
Then using (\ref{volbound}), the probability is at most $(k-1)(15/16)^{d/2}T$
which gives the desired bound.
\end{proof}

\begin{lemma}\label{lemma:E4bound}
For all $j$,
$\P(\cE_j^{(4)})
\leq (j-1)\left(\frac{15}{16}\right)^{d/2}T$
\end{lemma}
\begin{proof}
This case is very similar to Lemma \ref{lemma:E3bound}.
We wish to bound the probability that at step $j$, the
``supplementary" Poisson process of rate 1 in the set $A_j$ defined
by (\ref{Ajdef}) is non-empty. Using the same argument as above,
if the procedure has run successfully up to step $j$, then
each point $x_i, i<j$ is at distance at least $R/2$ from $x_j$.
Then the volume of the set in (\ref{Ajdef}) is at most
$(j-1)(15/16)^{d/2}T$.
\end{proof}

\begin{lemma}\label{lemma:treesize}
$\E\big(|\widetilde{V}_T^\downarrow(\emptyset)|\big)=e^T$, and hence
\begin{align}
\label{Markovbound}
\P\big(|\widetilde{V}_T^\downarrow(\emptyset)|>e^{2T}\big)&\leq e^{-T}.
\end{align}
\end{lemma}
\begin{proof}
For $k\geq 0$,
the expected number of descending paths $v_0, v_1,\dots v_k$
with weights less than $T$ in the PWIT, where $v_0$ is the origin, is given by
\[
\int_{0<t_k<\dots<t_1<T}dt_1\dots dt_k
\]
which is $T^k/k!$.
Since in the PWIT each point is the endpoint of at most one
such path, we can sum over $k$ to get
$\E\big(|\widetilde{V}_T^\downarrow(\emptyset)|\big)=e^T$.
The bound in
(\ref{Markovbound})
then follows by Markov's inequality.
\end{proof}

\begin{proof}[Proof of Proposition \ref{prop:couplingsuccess}]
Using the estimate in (\ref{Markovbound}), we can combine
the four previous bounds using a union bound.

If the procedure fails, then either it does so at step $j$ for some $j\leq e^{2T}$,
or $|\widetilde{V}_T^\downarrow(\emptyset)|>e^{2T}$. Then we can combine all the bounds
in Lemmas \ref{lemma:E1bound}, \ref{lemma:E2bound}, \ref{lemma:E3bound},
\ref{lemma:E4bound} and \ref{lemma:treesize}
to give
\begin{align*}
\P\left(
\bigcup_{j=1}^\infty \bigcup_{m=1}^4 \cE_j^{(m)}\right)
&\leq
\P\left(|\widetilde{V}_T^\downarrow(\emptyset)|>e^{2T}\right)
+ \P\left(\bigcup_{1\leq j\leq e^{2T}} \bigcup_{m=1}^4 \cE_j^{(m)}, |\widetilde{V}_T^\downarrow(\emptyset)|\leq e^{2T}\right)\\
&\leq
e^{-T}+\sum_{1\leq j\leq e^{2T}}\sum_{m=1}^4 \P\left(\cE_j^{(m)}, |\widetilde{V}_T^\downarrow(\emptyset)|\leq e^{2T}\right)
\\
&\leq e^{-T}+ 4\big(e^{2T}\big)^2\left(\frac{15}{16}\right)^{d/2}T^2
\end{align*}
(assuming $T\geq 1$).
For some constants $c$ and $\alpha$, this upper bound
is less than $e^{-\alpha T}$ for all $T\geq 1$ and
all $d>cT$,
as required for Proposition \ref{prop:couplingsuccess}.
\end{proof}

\section{Euclidean model: proof of Theorems
\ref{thm:asymmetricRd} and \ref{thm:symmetricRd}}
\label{sec:conclusionRd}

\begin{prop}\label{prop:larged}
Consider stable matching for the asymmetric two-type
model in $\R^d$ where each point is red with probability $1-\epsilon$
and blue with probability $\epsilon$.
Fix any $\delta>0$. Then
there exists $c'=c'(\delta)$ such that for
all small enough $\epsilon$,
and all $d>c'\frac1\epsilon e^{1/\epsilon}$,
the density of blue points which
remain unmatched is in
$[(1-\delta)\epsilon e^{-1/\epsilon+1},
(1+\delta)\epsilon e^{-1/\epsilon}+1]$.
\end{prop}

\begin{proof}
Recall from Section \ref{subsubsec:asymmPWIT}
that $r(t)$ and $b(t)$ are the probabilities
that the root of the PWIT is red (or respectively blue)
and is not matched within distance $t$.
We have $r(0)=1-\epsilon$ and $b(0)=\epsilon$,
and as $t\to \infty$, $r(0)\to 0$ and $b(0)\to b(\infty)=
\epsilon e^{-1/\epsilon+1}$.

Correspondingly, write $r^{(d)}(t)$
and $b^{(d)}(t)$ for the probability that,
in the Palm version of the model in $\R^d$,
the point at the origin is red (or respectively blue)
and is not matched within distance $f_d(t)$.
By ergodicity of the Poisson process, the sets of points in $\R^d$
which are red, or respectively blue, and unmatched
within distance $f_d(t)$ have densities $r^{(d)}(t)$
and $b^{(d)}(t)$ with probability 1.
Set also $b^{(d)}(\infty)=\lim b^{(d)}(t)$;
then the set of blue points which remain unmatched
for ever has density $b^{(d)}(\infty)$ with probability 1.

The density of blue points matched at distance
greater than $t$ cannot be greater than the
density of red points matched at distance greater than $t$
(by a mass transport argument). Hence we have that for any $t$,
\begin{equation}\label{bdinfbounds}
b^{(d)}(t)-r^{(d)}(t)
\leq
b^{(d)}(\infty)
\leq b^{(d)}(t).
\end{equation}

From (\ref{rb-equation}),
$b(t)-r(t)$ has positive derivative at all times, so that
$b(t)-b(\infty)\leq r(t)-r(\infty)=r(t)$ for all $t$.
Combining with the bound on $r(t)$ just after (\ref{rb-equation}),
we have that for all $t$,
\begin{equation}
b(t)-b(\infty)\leq r(t) \leq \frac{1}{t}.
\label{btbound}
\end{equation}

Now fix some $\gamma>1$, and let
$T=\gamma\frac1\epsilon e^{1/\epsilon}=\gamma e/b(\infty)>1$.
Suppose that $d>cT$, where $c$ is given by Proposition \ref{prop:couplingsuccess}.
We then have that
\begin{equation}
\left|r^{(d)}(T)-r(T)\right|+
\left|b^{(d)}(T)-b(T)\right|<e^{-\alpha T}.
\label{couplingbound}
\end{equation}

Combining all of (\ref{bdinfbounds}),
(\ref{btbound}), and
(\ref{couplingbound}), we get
\begin{align*}
b^{(d)}(\infty)
&\geq b^{(d)}(T)-r^{(d)}(T)\\
&\geq b(T)-r(T)
-\left|b^{(d)}(T)-b(T)\right|-\left|r^{(d)}(T)-r(T)\right| \\
&\geq b(\infty)-\frac{1}{T}-e^{-\alpha T},
\\
\intertext{and}
b^{(d)}(\infty)
&\leq b^{(d)}(T)\\
&\leq b(\infty) + (b(T)-b(\infty) +\left|b^{(d)}(T)-b(T)\right|\\
&\leq b(\infty)+\frac{1}{T}+e^{-\alpha T}.
\end{align*}
Since $T=\gamma e/b(\infty)$,
for given $\delta$
we can choose $\gamma$
sufficiently large that $b^{(d)}(\infty)$
lies in
$[(1-\delta)\epsilon e^{-1/\epsilon+1},
(1+\delta)\epsilon e^{-1/\epsilon}+1]$.
Taking $c'=c\gamma$ then completes the proof.
\end{proof}

\begin{proof}[Proof of Theorems \ref{thm:asymmetricRd}
and \ref{thm:symmetricRd}]
The statement of Theorem
\ref{thm:asymmetricRd} follows immediately from
Proposition \ref{prop:larged}.

A similar argument leads to Theorem \ref{thm:symmetricRd}
for the symmetric multi-type model.
Let $x_i^{(d)}(t)$ be the density of points of type $i$
which are not matched within distance $t$, and let
$x_1^{(d)}(\infty)$ be the density of points which
remain unmatched for ever.

As in Section \ref{subsubsec:symm}, define
$x_i(t)$ to be the probability that the
root of the PWIT has type $i$ and is not matched within distance $t$.
Then as $t\to\infty$, $x_i(t)\to 0$ for $i>1$ and
$x_1(t)\to x_1(\infty)>0$.

As at (\ref{bdinfbounds}) above, we have
\[
x_1^{(d)}(t)-\sum_{i>2} x_i^{(d)}(t)
\leq x_1^{(d)}(\infty)
\leq x_1^{(d)}(t).
\]
For any $\delta>0$, if $T$ is large enough, then
\[
x_i(T)<\delta/2k
\]
for all $i>1$, and also
\[
x_1(t)-x_1(\infty)<\delta/2k.
\]
Finally, again if $T$ is large enough, and $d>cT$
where $c$ is given by Proposition \ref{prop:couplingsuccess},
then
\[
\sum_i|x_i^{(d)}(T)-x_i(T)|<e^{-\alpha T} <\delta/2k.
\]
Combining all these bounds we obtain if
$T$ is large enough and $d>cT$, then
$|x_1^{(d)}(\infty)-x_1(\infty)|<\delta$.
Putting $\lambda=x_1(\infty)$,
this gives the conclusion of Theorem \ref{thm:symmetricRd}.
\end{proof}

Using the same methods, one can obtain a result on 
stable matching with a general compatibility graph.
Consider again a stable matching model for a Poisson 
process of rate $1$ in $\R^d$. Let $V$ be the set of 
points of the process, and let each point of $V$
independently receive colour $i$ with probability $p_i$, for $1\leq i\leq k$, where $(p_1,\dots,p_k)$ is a probability vector.

As in Section \ref{sec:uniqueness},
let $\mathcal{G}$ be an undirected graph with vertex set $\{1,\ldots,k\}$ with no parallel edges but possibly with self-loops.  An edge between $i$ and $j$ indicates that colours $i$ and $j$ are compatible,
i.e.\ that points of colours $i$ and $j$ are allowed to be matched to each other; in this case write 
$i\sim j$.

For two points $x,y\in V$ of respective colours $i$ and $j$,
let $\ell(x,y)=|x-y|$ if $i\sim j$, and $\ell(x,y)=\infty$ 
otherwise. 

For the corresponding stable matching model on the PWIT
(with the same vector of colour probabilities $(p_1,\dots,p_k)$), $\ell$ is defined as explained at the beginning of Section \ref{sec:PWITmodel}.
Just as in Section \ref{sec:PWIT}, 
we can consider the probability $x_i(t)$ that the root has colour $i$ and is not matched along an edge of weight less than
$i$. We have $x_i(0)=p_i$, and 
\begin{equation}\label{newxi}
x_i'(t)=-x_i(t)\sum_{j\sim i}x_j(t)
\end{equation}
for each $i$. 
Then $x_i(\infty)=\lim_{t\to\infty}x_i(t)$ gives the
probability that the root has colour $i$ and is unmatched. 

\begin{thm}\label{thm:meta}
Fix a probability vector $(p_1, \dots, p_k)$. 
Consider a multi-type stable matching of
a Poisson process of rate $1$ in $\R^d$ with $k$ colours,
where $p_i$ gives the probability of colour $i$, 
with compatibility graph $\mathcal{G}$.
As $d\to\infty$, the intensity of unmatched colour-$i$ points converges to the value $x_i(\infty)$ obtained from (\ref{newxi}),
which is the probability in the corresponding stable matching
model on the PWIT that the root has colour $i$ and is unmatched. 
\end{thm}
The proof of Theorem \ref{thm:meta} can be done using
exactly the same approach described in the proof of Theorems 
\ref{thm:asymmetricRd} and \ref{thm:symmetricRd} above.

\section{Hierarchical model: proof of Theorem \ref{thm:hierarchical}}
\label{sec:conclusionhierarchical}
Recall that in the context of Theorem \ref{thm:hierarchical},
we have a Poisson process of rate $\lambda$ on $\R_+$,
in which each point is coloured blue with probability $p$
and red with probability $1-p$. Red-red and red-blue matches
are allowed, but not blue-blue. Given the hierarchical distance
$\rho$ defined by (\ref{rhodef}),
let $\ell(x,y)=\rho(x,y)$, except when both points $x$ and $y$ are blue,
in which case $\ell(x,y)=\infty$.

Now we cannot apply Proposition \ref{prop:uniquestable} directly,
since with probability 1 there will be points which are
equidistant from others, and so condition (i) does not hold,
and the stable matching for $\ell$ will not be unique.
We can consider instead the distance $\trho$ on $\R$
given by
\[
\trho(x,y)=\rho(x,y)+|x-y|.
\]
Correspondingly, for two Poisson points $x$ and $y$,
let $\tell(x,y)=\ell(x,y)+|x-y|$ (so that
$\tell(x,y)=\trho(x,y)$ unless
both points are blue, in which case $\tell(x,y)=\infty$).
Now with probability 1, the function $\tell$ does satisfy
the conditions of Proposition \ref{prop:uniquestable}, so
that there exists a unique stable matching for $\tell$.

One can easily show that if $\tell(u,v)\leq \tell(x,y)$,
then also $\ell(u,v)\leq \ell(x,y)$.
Using the definition of stable matching, it follows
that if $M$ is
stable for $\tell$, then it is also stable for $\ell$.
So at least one stable matching for $\ell$ exists.

Recall that we write $N_k(m)$ for the excess of blue points
over red points in the interval $[2^k m, 2^k (m+1)]$ out
of those which are not matched to another point in the interval,
i.e.\ which are not matched at distance $k$ or less.
We use the recursion at (\ref{hierarchicalrecursion}) for
the distribution of $N_k(m)$ as $k$ varies.
Define
\begin{align*}
\beta_k&=\P(N_k(m)\text{ is even})\\
\gamma_k&=\P(N_k(m)\text{ is odd and positive})\\
\delta_k&=\P(N_k(m)\text{ is even and positive})
\end{align*}
First note that for any $k$,
\begin{align*}
\beta_{k}&=\beta_{k-1}^2+\left(1-\beta_{k-1}\right)^2\\
&\geq 1/2.
\end{align*}
Then
\begin{align*}
\gamma_{k+1}&=2\gamma_{k}\beta_k
+ 2\delta_{k}(1-\beta_{k}-\gamma_{k})
\\
&\geq \gamma_{k}
\\
\intertext{and}
\delta_{k+1}&\geq\gamma_k^2,
\\
\intertext{so that}
\gamma_{k+2}
&= 2\gamma_{k+1}\beta_{k+1}
+2\delta_{k+1}(1-\beta_{k+1}-\gamma_{k+1})
\\
&\geq 2\gamma_k\beta_{k+1}
+2\gamma_k^2(1-\beta_{k+1}-\gamma_{k+1})
\\
&=
2\gamma_k \frac12 + 2(\gamma_k-\gamma_{k}^2)(\beta_{k+1}-\frac12)
+2\gamma_k^2(\frac12-\gamma_{k+1})
\\
&\geq
2\gamma_k\frac12 +2\gamma_k^2(\frac12-\gamma_{k+1})\\
&\geq
\gamma_k + \frac{\gamma_{k}^2}{3}
\end{align*}
as long as $\gamma_{k+1}\leq 1/3$.

So $\gamma_k$ will eventually reach $1/3$ for some $k$.
(In fact, for any constant $c$, the
number of iterations of the recursion
$x\to x+x^2/3$ required to exceed the value $c$ starting from
the value $x_0$ is $3(1+o(1))x_0^{-1}$ as $x_0\to 0$.)

Here we have $\gamma_0>
\lambda e^{-\lambda}p
$,
since this is the probability that an interval of length 1
contains no red points and exactly one blue point.
Then for some function $k_0=k_0(\lambda, p)$,
we have that $\gamma_k\geq 1/3$ for all $k\geq k_0(\lambda, p)$.
Then also $E N_k(m)\geq1/6$ for $k\geq k_0(\lambda, p)$,
since $\P(N_k(m)\geq 1)\geq 1/3$, while
$\P(N_k(m)=-1)=\P(N_k(m) \text{ odd})-\gamma_k\leq 1/2-1/3=1/6$.

Note that, more weakly than (\ref{hierarchicalrecursion}),
we have that $N_k(m)\geq N_{k-1}(2m)+N_{k-1}(2m+1)$.

In particular, for any $r>0$, the quantity $N_{k_0+r}(m)$
is bounded below by a sum of $2^r$ independent copies
of the random variable $N_{k_0}(0)$, which has finite mean and is bounded below.
Hence, using standard large deviations results, there is some constant $\theta>0$ such that for all $k\geq k_0$,
\[
\P(N_k(m)< \theta 2^k)\leq e^{-\theta 2^k}.
\]
In particular, the sum of the right-hand side of all $k\geq k_0$ is finite. We obtain
that with probability 1, there exists some $K$ such that
\begin{equation}\label{allblue}
N_K(0)\geq \theta 2^K, \text{ and } N_k(1)\geq \theta 2^k \text{ for all } k\geq K.
\end{equation}

The quantities $N_K(0)$ and $\{N_k(1): k\geq K\}$ relate
to the intervals $[0,2^K)$ and $\{[2^k, 2^{k+1}): k\geq K\}$
which form a partition of $\R_+$. If indeed (\ref{allblue})
holds, then (for any stable matching) none of these intervals contains
a red point which is matched outside the interval. Hence, for any
of these intervals, all the blue points which are not matched within the interval
are not matched at all.

Hence in fact the number of unmatched blue points in $[0,2^k)$ is at least
$\theta 2^k$ for all $k\geq K$. Taking $R=2^K$ and $c=\theta/2$ then gives the result of Theorem \ref{thm:hierarchical}.

\section{Existence and uniqueness of a stable matching:
proof of Propositions \ref{prop:uniquestable}, \ref{prop:whoismatched} and \ref{prop:stablerescaled}}
\label{sec:stableproof}

Before proving Propositions \ref{prop:uniquestable} and
\ref{prop:whoismatched}, we first note a useful characterisation
of a stable matching which holds under the assumption
that the weights of edges from any given point $x$ are
all distinct.

For convenience we repeat here the definition
given at
(\ref{stablecondition}); a matching $M$
is stable if
\begin{equation}
\label{stablecondition2}
\ell(x,y)\geq\min\left(d_M(x), d_M(y)\right)
\text{ for all $x$ and $y$.}
\end{equation}

\begin{lemma}\label{lemma:stableprop}
Suppose condition (i) of Proposition \ref{prop:uniquestable}
holds (the distinct weights condition). Then a matching $M$
is stable iff for all $x\in V$ and $R\in(0,\infty)$,
\begin{equation}
\label{dMprop}
\{d_M(x)<R\} \Leftrightarrow \{\exists y \text{ such that }
\ell(x,y)<R \text{ and } d_M(y)\geq \ell(x,y)\}.
\end{equation}
\end{lemma}
\begin{proof}
We will show that if the distinct weights condition holds,
then (\ref{stablecondition2}) and (\ref{dMprop}) are equivalent.

Suppose the matching is not stable, so (\ref{stablecondition})
fails for some $x,y$. Then we can
choose $R$ with $\ell(x,y)<R<\min\{d_M(x), d_M(y)\}$,
in which case the right
side of (\ref{dMprop}) holds but the left side does not.
Hence (\ref{dMprop}) also fails.

On the other hand, suppose that (\ref{stablecondition2}) holds.

If $d_M(x)\geq R$ but the right side of (\ref{dMprop}) were true,
then $x$ and $y$ are not partners (since $\ell(x,y)<d_M(x)$).
But by the distinct weights condition, $y$ does not have a partner
$z$ with $\ell(z,y)=\ell(x,y)$. Hence in fact $d_M(y)>\ell(x,y)$
strictly; then $\ell(x,y)<\min\{d_M(x), d_M(y)\}$
contradicting (\ref{stablecondition2}).

Meanwhile if $d_M(x)<R$ then $x$
has a partner $y$ with $\ell(x,y)<R$ and $d_M(y)=d_M(x)$,
so the right side of (\ref{dMprop}) holds.
Hence under (\ref{stablecondition2}), the
left and right sides of (\ref{dMprop}) are equivalent, as required.
\end{proof}

\begin{proof}[Proof of Propositions \ref{prop:uniquestable}
and \ref{prop:whoismatched}]
The underlying idea is essentially the same
as was used
in \cite{HPPS} for the special cases of one-type
and symmetric two-type models with weights given by distances
in $\R^d$. However we will present the construction rather differently,
so as to make explicit the way in which the stable matching is determined by the collections of descending paths as stated in
Proposition
\ref{prop:whoismatched}. (The argument in \cite{HPPS} is
phrased in terms of  the following recursive construction. Call two
points $x$ and $y$
\textit{mutually closest} if $\ell(x,y)<\ell(x,z)$ for all $z\ne y$
and $\ell(x,y)<\ell(z,y)$ for all $z\ne x$. Now, given the point
configuration, match all mutually closest pairs of points to each other, and then remove them from the configuration. Now match all
pairs which are mutually closest in the remaining set of points;
repeat indefinitely. Lemma 15 of \cite{HPPS} shows that
for the models under consideration, this
recursive construction yields a stable matching, which is
in fact unique.)

We begin by justifying the first assertion of
Proposition \ref{prop:whoismatched}.
Consider the subgraph spanned by the edges of $E_R^{\downarrow}(x)$,
obtained by taking the union of
all descending paths from $x$ with weights less than $R$.
From condition (ii), any vertex has finite degree in this
subgraph (since any vertex is incident to only finitely
many edges with weight less than $R$). If there were
infinitely many vertices in this subgraph, then there
would be arbitrarily long descending paths from $x$,
and then, by compactness, an infinite descending path.
But this is excluded by condition (iii), so indeed
$E_R^{\downarrow}(x)$ is finite.

Now we argue that in fact we
can determine whether $d_M(x)$ by inspecting the set
$E^{\downarrow}_R(x)$, as required for Proposition \ref{prop:whoismatched}.
Using Lemma \ref{lemma:stableprop}, it is enough to determine
whether $d_M(y)< \ell(x,y)$ for all $y$ with $\ell(x,y)<R$.

We proceed by induction on the size of
$E^{\downarrow}_R(x)$. Consider $y$ with $\ell(x,y)<R$.
Then $E_{\ell(x,y)}^{\downarrow}(y)\subset E_R^{\downarrow}(x)$ (the inclusion is strict
since the edge $(x,y)$ is included in the second set but not in
the first).
So for the induction step, we may assume that
we know whether $d_M(y)<\ell(x,y)$ for each such $y$,
and hence indeed we can deduce whether $d_M(x)<R$.
The base of the induction is the case where $x$ is incident
to no edges of weight less than $R$, in which case certainly
$d_M(x)\geq R$.

This completes the proof of Proposition \ref{prop:whoismatched}.
The weights of edges in $E_R^{\downarrow}(x)$ determine whether
$d_M(x)<R$, and hence the weights of all the edges determine
the value of $d_M(x)$ and so, since the weights of edges
incident to $x$ are distinct by (i), in fact determine $M(x)$.
Hence there is at most one stable matching.

To show the existence of a stable matching, note
that the inductive procedure above can be used to \textit{define}
a function $d_M(x)$ for $x\in V$ which satisfies
(\ref{dMprop}). We need to show that this function
actually corresponds to a matching $M$ in the sense of (\ref{dMdef}).

Suppose $d_M(u)=s<\infty$. Then applying (\ref{dMprop})
with $x=u$ and considering both
$R\leq s$ and $R>s$, we obtain that
for some $v$, $\ell(u,v)=s$ and $d_M(v)\geq s$.
Then in turn we can apply (\ref{dMprop}) with $x=v$
and any $R>s$; because we have $u$ with $\ell(v,u)<R$ and
$d_M(u)\geq \ell(u,v)$ it follows that $d_M(v)<R$,
and hence in fact $d_M(v)=s$. Thus there is a point $v$
satisfying $\ell(u,v)=d_M(v)=s$, and this point is unique
by condition (i). Then define $M(u)=v$ (and similarly $M(v)=u$).

Meanwhile if $d_M(u)=\infty$, define $M(u)=u$.
Then indeed $M$ is a matching, and satsfies (\ref{dMprop}),
and so is stable.

Further, suppose $x$ and $y$ are both unmatched by $M$,
so that $d_M(x)=d_M(y)=\infty$. If $\ell(x,y)$ were finite,
then for any $R>\ell(x,y)$, again the right side of (\ref{dMprop})
would hold but the left side would not. Hence indeed $\ell(x,y)=\infty$ as required for the final statement of Proposition
\ref{prop:uniquestable}.

Finally, in Proposition \ref{prop:stablerescaled}
there is a bijection from $V_R^{\downarrow}(x)\subset V$
to $\widetilde{V}^{\downarrow}_{f(R)}(\widetilde{x})\subset \widetilde{V}$
which maps $x$ to $\widetilde{x}$ and under which
the edge-weights are related by the strictly increasing function $f$. The definition of stable matching,
and the equivalent condition in (\ref{dMprop}),
use only information about relative orderings of edge-weights;
such orderings are preserved when the edge-weights
are rescaled by $f$.
Hence the inductive procedures
for determining whether $d_M(x)<R$ for the stable matching
$M$ of $V$, and whether $d_{\widetilde{M}}(\widetilde{x})<f(R)$
for the stable matching $\widetilde{M}$ of $\widetilde{V}$,
proceed identically, and so indeed $d_M(x)<R$ if and only if
$d_{\widetilde{M}}(\widetilde{x})<f(R)$.
\end{proof}

\section*{Open Problems}
Consider a homogeneous Poisson process in $\R^d$ in which each point is independently assigned a colour according to a fixed probability vector.
\begin{enumerate}[(i)]
\item For the asymmetric two-type stable matching (with
    only red-blue and red-red matches allowed), do there
    exist red and blue probabilities for which all points
    are matched?  The question is open for every $d\geq
    1$.
\item For the asymmetric two-type stable matching in a
    fixed dimension $d$, is the intensity of
    unmatched blue points non-decreasing in the initial
    probability of blue points?  Is it strictly
    increasing?
\item For the symmetric three-type stable matching (where
    points of any two distinct colours are allowed to
    match),  suppose that the probabilities $p_2,p_3$ of
    two of the colours are equal.  Symmetry and
    ergodicity imply that either all points are matched,
    or only points of color $1$ are unmatched.  Are all
    points matched when $p_1<p_2=p_3$?
	This question is open for all $d$.
    Are some points unmatched when $p_1>p_2=p_3$? The methods
    of this paper can be used to show that the answer is yes for
    large enough $d$, but for small $d$ the question is open.
\item More generally, for which matching restrictions, probability vectors, and dimensions are all points matched?
\item Can the PWIT provide information about matching distance in high dimensions?  For example, in the case of two-color stable matching (where only red-blue matches are allowed), with equal probability of red and blue points, the probability for a typical point to be matched at distance at least $r$ is known \cite{HPPS} to be between $r^{-\alpha}$ and $r^{-\beta}$ as $r\to\infty$ where $\alpha(d),\beta(d)\in(0,\infty)$, but the bounds on these constants are far apart except when $d=1$. What can be said about their asymptotic behaviour as $d\to\infty$?
\end{enumerate}


\section*{Acknowledgments}
We thank Robin Pemantle for helpful conversations at an early stage of this work. We thank a referee for several valuable comments and suggestions. JBM thanks the Theory Group of Microsoft Research for their support and hospitality; this work was carried out while he was a visiting researcher.

\bibliography{PWITnew}
\bibliographystyle{abbrv}

\bigskip

\noindent
\textsc{Alexander E.\ Holroyd}
\\
\textit{E-mail address:}
\texttt{holroyd@uw.edu}
\\
\textit{URL:}
\texttt{http://aeholroyd.org}
\medskip

\noindent
\textsc{James B.\ Martin, University of Oxford, UK
}
\\
\textit{E-mail address:}
\texttt{martin@stats.ox.ac.uk}
\\
\textit{URL:}
\texttt{http://www.stats.ox.ac.uk/\textasciitilde{}martin}

\medskip

\noindent
\textsc{Yuval Peres}
\\
\textit{E-mail address:}
\texttt{yuval@yuvalperes.com}
\\
\textit{URL:}
\texttt{http://yuvalperes.com}


\end{document}